\documentclass[12pt,leqno,twoside]{amsart}

\usepackage[latin1]{inputenc}
\usepackage[T1]{fontenc}
\usepackage[colorlinks=true, pdfstartview=FitV, linkcolor=blue, citecolor=blue, urlcolor=blue]{hyperref}
\usepackage{amstext,amsmath,amscd, bezier,indentfirst,amsthm,amsgen,enumerate, geometry}
\usepackage[all,knot,arc,import,poly]{xy}
\usepackage{amsfonts,color, soul}  
\usepackage{amssymb}
\usepackage{latexsym}
\usepackage{epsfig}
\usepackage{graphicx}
\usepackage{srcltx}
\usepackage{enumitem}

\newtheorem{theorem}{Theorem}[section]
\newtheorem{corollary}[theorem]{Corollary}
\newtheorem{lemma}[theorem]{Lemma}

\newtheorem{conjecture}[theorem]{Conjecture}

\newtheorem{proposition}[theorem]{Proposition}
\theoremstyle{definition}
\newtheorem{definition}[theorem]{Definition}
\theoremstyle{remark}
\newtheorem{remark}[theorem]{\sc Remark}
\newtheorem{example}[theorem]{\sc Example}

\newcommand{\proj}{{\rm{proj}}}

\newcommand{\Sing}{{\rm{Sing\hspace{2pt}}}}

\newcommand{\Disc}{{\rm{Disc\hspace{2pt}}}}

\newcommand{\re}{{\rm{Re}}}

\newcommand{\im}{{\rm{Im\hspace{2pt}}}}

\newcommand{\dd}{{\rm{d}}}

\newcommand{\e}{\varepsilon}
\newcommand{\m}{\setminus}

\newcommand{\cC}{{\mathcal C}}

\newcommand{\bR}{{\mathbb R}}
\newcommand{\bC}{{\mathbb C}}

\begin{document}

\title[Geometrical conditions for the existence of a Milnor vector field]{Geometrical conditions for the existence of a Milnor vector field}

\author{Raimundo Araújo dos Santos, Maico F. Ribeiro}



\thanks{The authors acknowledge the Professor Mihai Marius Tib\u ar for the valuable suggestions.}

\begin{abstract}
 
We introduce several sufficient conditions to guarantee the existence of the Milnor vector field for new classes of singularities of map germs. This special vector field is related with the equivalence problem of the Milnor fibrations for real and complex singularities, if they exit. 
\end{abstract}

\maketitle

\section{Introduction}

For the holomorphic functions germs  $f: (\mathbb{C}^{n+1}, 0) \to (\mathbb{C}, 0)$ with $\dim \Sing f \geq 0$,  Milnor showed that for any small enough  $ \e >0$  the restriction map
 \begin{equation}\label{csf}
 f/\| f\|: S^{2n+1}_\e \m K_\e \to S_{1}^1
 \end{equation}
 is a locally trivial smooth fibration, where $K_\e:=S^{2n+1}_\e \cap f^{-1}(0)$.

 It was proved later by Lê D. Tráng in \cite{Le} that for any small enough $\epsilon>0$, there exists $0 < \delta \ll \epsilon$, such that the restriction map $
 f_{\mid}:  {B}_{\e}^{2n+2} \cap f^{-1} \left(D_{\delta} \m \{0\} \right) \rightarrow  D_{\delta}\m \{0\} $
 is the projection of a locally trivial smooth fibration, where $ {B}_{\e}^{2n+2} $ and $ D_{\delta} $ are  the open balls in the respective spaces. It is well known that the previous fibration  induces the smooth fibration
 \begin{equation}\label{mfe3}
 f/\| f\|: B_{\e}^{2n+2} \cap f^{-1} \left(S^{1}_{\eta}\right) \rightarrow  S^{1}_{1}.
 \end{equation}

  Moreover, it follows from  the Milnor \cite{Mi} and from Lê \cite {Le} works that  the fibrations (\ref{csf}) and (\ref{mfe3}) are equivalent.
 
In the case of  real analytic map germ  $G:(\bR^{m},0) \rightarrow (\bR^{p}, 0)$, $ m>p\ge 2 $ with $ \Sing G \subset G^{-1}(0) $ as a set germ, the authors of \cite{Ma,PT} gave several conditions that ensure the existence of the \emph{empty tube fibration}
$$ G/\| G\|: B_{\e}^{m} \cap G^{-1} \left(S^{p-1}_{\eta}\right) \rightarrow  S^{p-1}_{1}.$$

Under special conditions, in the papers \cite{dST0, dST1, ACT, CSS1,PS0,PS2,PS3} it was proved the existence of the \emph{sphere fibration}
$$ G/\| G\|: S^{m-1}_\e \m K_\e \to S^{p-1}_1. $$

More recently, in \cite{ART} and \cite{ART2}, the authors extended all previous  results for the case when the singular set $\Sing G$ has positive dimension and is not necessarily included in the central fibre $G^{-1}(0)$, i.e., when $ \Disc G$ has positive dimension. However, even in the case where $ \Sing G = {0} $ it is not known whether or not these two fibrations are equivalent. 

This equivalence problem has been approached by several authors in the last years, first in the case $ \Disc G = \{0\} $, e.g., \cite{A1,dST0,dST1,Oka3,Ha,CSS1}, and more recently in \cite{ART2} in the more general case where $\Disc G $ has positive dimension. 

In the paper \cite{ART2}, this problem was  formulated in general as a conjecture, see Conjecture \ref{conjec}, Section 3. Moreover, the authors showed the relationship between the equivalence problem and the existence of the so called \emph{Milnor vector field}. Beside that, it was also shown that under a topological condition under the so called \emph{Milnor set} (see \cite[Corollary 4.11]{ART2}) this conjecture is solved. 

\vspace{0.2cm}

 The aim of this paper is to introduced several new sufficient conditions to ensure the existence of the Milnor vector field and hence to ensure the equivalence between the tube and sphere fibrations. Moreover, we present plenty of new classes of real and complex singularities satisfying our conditions.

The paper is organized as follows. In section 2, we remind the main results about the existence of the local tube and sphere fibration structures in the general setting as proved in \cite{ART,ART2}. In Section 3, we introduce the {\it Milnor vector field} as previously defined in \cite{ART2} and connect it with the equivalence problem. In section 4, we introduce new criteria which ensure the existence of the Milnor vector field. In last section, as application,  we introduced new classes of real and complex singularities for which  the Conjecture \ref{conjec} holds true. Among them, we point out the \emph{mixed simple \L-maps} which are a special type of  mixed singularities.

\section{Fibrations structures}
As explained in the paper \cite{ART}, given a non-constant analytic map germ \linebreak $G:(\bR^{m},0) \rightarrow (\bR^{p}, 0)$, $m\ge p >0$, the set $\Sing G$ is well-defined as a germ of set on the source space. However, in general, in the target space the germ of sets $G(\Sing G)$ and $\im G$ do not. We make it clear in the next definition.

\begin{definition}\cite{ART}\label{d:nice}
	Let $G:(\bR^{m},0) \rightarrow (\bR^{p}, 0)$, $m\ge p >0$, be a continuous map germ.
	We say that the image  $G(K)$  of a set $K\subset \bR^{m}$ containing $0$  is a \emph{well-defined set germ}
	at $0\in \bR^{p}$ if for any open balls $B_{\e},  B_{\e'}$ centred at 0, with  $\e, \e' >0$, we have the equality of germs
	$[G(B_{\e}\cap K)]_{0} =  [G(B_{\e'}\cap K)]_{0}$.
	
	Whenever the images $\im G$ and  $G(\Sing G)$ are well-defined as germs,  we say that $G$ is a \emph{nice map germ}.
\end{definition}

In \cite{ART} the authors found sufficient conditions to an analytic map germ to be a nice germ and have introduced a good class of maps with this property, namely the map germ of type $ f\bar{g}:(\bC^n,0) \to (\bC,0) $ where $  f,g:(\bC^n,0) \to (\bC,0) $ are holomorphic germs such that the meromorphic function $ f /g $ is irreducible. Moreover, they considered the \emph{discriminant} of a nice map germ $G$, $\Disc G$, as the locus where the topology of the fibers may change, which is a closed subanalytic germ of dimension strictly less than $p$. 

\subsection{Existence of the tube fibration}

\begin{definition}\cite{ART}\label{d:tube}
	Let $G:(\bR^{m},0) \rightarrow (\bR^{p}, 0)$, $m\ge p>0$,  be a non-constant nice analytic map germ.
	We say that $G$ has {\em  Milnor-Hamm tube fibration} 
	if, for any $\e > 0$ small enough, there exists  $0<\eta \ll \e$ such that the restriction:
	\begin{equation}\label{eq:tube}
	G_| :  B^{m}_{\e} \cap G^{-1}( B^{p}_\eta \m \Disc G) \to  B^{p}_\eta \m \Disc G
	\end{equation}
	is a locally trivial fibration over each connected component $\cC_{i}\subset B^{p}_\eta \m \Disc G$, such that it is independent of the choices of $\e$ and $\eta$ up to diffeomorphisms. 
\end{definition}

Let $U \subset \mathbb{R}^m$ be an open set,  $0\in U$,  and let $\rho:U \to \mathbb{R}_{\ge 0}$ be a non-negative proper function which defines the origin, for instance the Euclidean distance. 
 
We consider here the following definition:	

\begin{definition}\label{d:ms}
	Let 	$G:(\mathbb{R}^m, 0) \to (\mathbb{R}^p,0)$ be a non-constant nice analytic map germ. The set germ at the origin:
	\[M_\rho(G):=\overline{\left\lbrace x \in U \mid \rho \not\pitchfork_x G \right\rbrace } \]
	is called the set of \textit{$\rho$-nonregular points} of $G$, or the \emph{Milnor set of $G$}.
\end{definition}

In this paper we will consider only the Euclidean distance function $\rho_{E}$, and we write for short $M(G):= M_{\rho_{E}}(G)$.

\vspace{0.2cm}

The following condition was used in \cite{ART} and \cite{ART2} to ensure the existence of the Milnor-Hamm fibrations.
\begin{equation}\label{eq:main}
\overline{M(G)\m G^{-1}(\Disc G)}\cap V_G \subseteq \{0\}.
\end{equation}

\begin{proposition} \cite[Lemma 3.3]{ART}\label{t:tube}
	Let $G:(\bR^m, 0) \to (\bR^p,0)$ be a non-constant nice analytic map germ, $m\ge p >0$.  If $G$ satisfies condition \eqref{eq:main}, then $G$ has a Milnor-Hamm tube fibration \eqref{eq:tube}.
\end{proposition}

\begin{corollary} [Existence of the Milnor-Hamm empty tube fibration] \label{tube} Let \linebreak $G:(\bR^m, 0) \to (\bR^p,0)$ be a non-constant nice analytic map germ. If $G$ satisfies condition \eqref{eq:main}, then for all small enough $\e > 0$, there exists  $0<\eta \ll \e$ such that the restriction:
	\begin{equation}\label{eq:emptytube}
	G_| :  B^{m}_{\e} \cap G^{-1}( S^{p-1}_\eta \m \Disc G) \to  S^{p-1}_\eta \m \Disc G
	\end{equation} 
	is a locally trivial smooth fibration.
\end{corollary}

\subsection{Existence of the sphere fibration}

Several author have worked on the problem of fibration over spheres in the real settings, for isolated and non-isolated singularities, e.g. \cite{S1, RSV, PS0, CSS1, dST0, dST1, PS2, PS3, A, RA,ACT}. In the recent paper \cite{ART2}  the authors generalized all previous results as we describe below.

	\begin{definition}\cite{ART2}
		Let $G:(\mathbb{R}^{m},0)\to (\mathbb{R}^{p},0)$ be a nice analytic map germ. We say that its discriminant $\Disc G$ is \emph{radial} if, as a set germ at the origin, it is a union of real half-lines or just the origin. 
	\end{definition} 
	
	
	
	Let  $G:U\to \mathbb{R}^{p}$ be a representative of the nice map germ $G$ for some open set $U \ni 0$.  We consider the map:
	\begin{equation} \label{eq:spherefib}
	{\Psi_G} := \frac{G}{\| G\|}: U \m V_G \to  S^{p-1}.
	\end{equation}

	As we have seen in \cite{ART2}, if $\Disc G$ is radial then the restriction map  
	\begin{equation}\label{ssf}
	\displaystyle{{\Psi_G}_{|} : S_{\epsilon}^{m-1}\m G^{-1}(\Disc G )\to S^{p-1}\m \Disc G }
	\end{equation}
	is well defined for small enough $\epsilon >0$, in this case we say that $ G $ has a \textit{Milnor-Hamm sphere fibration} whenever \eqref{ssf} is a locally trivial smooth fibration which is independent, up to diffeomorphisms, of the choice of $\e$ provided it is small enough.
	
	We consider the Milnor set $M({\Psi_G})$ of the map \eqref{eq:spherefib}, i.e. the germ at the origin of the $\rho$-nonregular points of ${\Psi_G}$. 
	We  say that  \textit{$ \Psi_G$ is $\rho$-regular} if:  
	\begin{equation}\label{eq:reg}
	M ({\Psi_G}) \m G^{-1}(\Disc G)=\emptyset.
	\end{equation}
	
	\begin{theorem}\cite[Theorem 3.5]{ART2}\label{sf2} 
		Let $\displaystyle{G:(\mathbb{R}^{m},0)\to (\mathbb{R}^{p},0)}$, $m>p\ge2$, be non-constant nice analytic map germ with radial discriminant,  satisfying the condition \eqref{eq:main}. If ${\Psi_G}$ is $\rho$-regular then $G$ has a  Milnor-Hamm sphere fibration.
	\end{theorem}
	
	In the case where $\Disc G=\{0\}$ the condition \eqref{eq:main} becomes $$\overline{M(G)\m V_{G}}\cap V_G \subseteq \{0\}$$ which has been used in \cite{Ma, dST1, ACT} in order to guarantee the existence of the tube and sphere fibrations. It is also important to point out that under the conditions \eqref{eq:main} and $\Disc G=\{0\}$ the map germ $G$ is nice, see \cite[Corollary 4.7]{Ma}.
	
\subsection{Fibrations equivalence problem}

	As has been seen in \cite{ART}, if the discriminant set $\Disc G$ has positive dimension it intersects all spheres of small enough radius  in the target space. However, the condition \eqref{eq:main} ensure that the restriction \eqref{eq:emptytube}
	is also the projection of a locally trivial smooth fibration.
	
	If $\Disc G$ is radial,  the fibration \eqref{eq:emptytube} may be composed with the canonical projection $\pi:= s/\|s\|: \mathbb{R}^{p}\m \{0\} \to S^{p-1}$ to get a locally trivial smooth fibration
	\begin{equation} \label{sfibration1}
	{\Psi_G}_{|}:B^{m}_\e \cap G^{-1}(S^{p-1}_\eta \m \Disc G) \to S^{p-1}\m  \Disc G.
	\end{equation}
	Moreover, under additional hypothesis of  $\rho$-regularity for $ \Psi_G $, the Theorem \ref{sf2} ensure that the restriction \eqref{ssf} is a Milnor-Hamm sphere fibration. Therefore, the following question becomes natural:
	
	\centerline{\textit{When are the fibrations \eqref{ssf} and \eqref{sfibration1} equivalent?}}
	
	In order to study the equivalence of fibrations in this setting, the authors in \cite{ART2} have adapted the Milnor's method of ``blowing away the tube to the sphere'' introduced in \cite{Mi}, which uses a special vector field 
	as we remind below.

\section{The Milnor Vector Field}

	\begin{definition}\cite{ART2}\label{def-gvf}
		Let $ G:(\bR^m,0) \to (\bR^p,0) $,$ m>p\ge 2 $ be an analytic map germ with radial discriminant. One calls  \textit{Milnor vector field}, abbreviated MVF, a vector field $\nu$ which satisfies the following conditions for any $x\in B^{m}_{\e} \m G^{-1}(\Disc G)$:
		\begin{enumerate}
			\item [($c_1$)] $\nu(x)$ is tangent to the fiber $X_{y}$, where $y=\Psi_G(x)$,
			\item [($c_{2}$)] $\left\langle \nu (x), \nabla \rho(x) \right\rangle >0$, 
			\item [($c_{3}$)] $\left\langle \nu (x), \nabla \|G(x)\|^{2} \right\rangle >0$.
		\end{enumerate}
	\end{definition}
	
	\begin{remark}
		The existence of a MVF  does not insure the existence of the fibrations \eqref{ssf} or \eqref{sfibration1}. Actually, the Milnor-Hamm fibration may not even be well-defined as shown by the Hansen's example $ G(x,y,z)=\left(x^2+y^2,(x^2+y^2)z\right) $  which is not nice.  \end{remark}
	
	As pointed out in \cite{ART2}, the existence of MVF guarantees the equivalence between the Milnor-Hamm fibrations provided that they exist, as shows the next result.
	
	\begin{theorem}\cite[Theorem 4.2]{ART2}\label{tm} 
		Let  $G:(\mathbb{R}^{m},0)\to (\mathbb{R}^{p},0)$, $m>p\geq 2$ be an analytic nice map germ with  radial discriminant, such that  the Milnor-Hamm tube and sphere fibrations exist.
		If  a Milnor vector field exists, then these fibrations are equivalent.
	\end{theorem}
	
	It is important notice, however, that in the real setting the existence of MVF is no insured in general, like the complex case. This existence problem has been approached by several authors in the last years, first in the case $ \Disc G = \{0\} $ e.g., \cite{A1,dST0,dST1,Oka3,Ha,CSS1}, and more recently in \cite{ART2}.  In each setting, the authors produced sufficient conditions. However, up to now, there is no  clear proof for this problem in its complete generality,  and no counterexample has been found yet \footnote{As have been point out in \cite{ART2},  the existence proof attempt  in \cite{CSS1} for case $ \Disc G = \{0\} $ appears to contain a non-removable gap. For more details, see \cite[Chapter 7]{R}}.  The following conjecture explains what we mean by ``complete generality'': 
	\begin{conjecture}\cite{ART2}\label{conjec}
		Let $G:(\mathbb{R}^{m},0)\to (\mathbb{R}^{p},0)$, $m>p\geq 2$ be an analytic nice map germ with  radial discriminant and such that ${\Psi_G}$ is $\rho$-regular. If both fibrations  \eqref{ssf} and \eqref{sfibration1} exist for any small enough $\e>0$ and $0<\eta\ll \epsilon $, then  they are equivalent.
	\end{conjecture}
	
	
	
	For the sake of simplicity, given $x\in B_{\e}^{m} \m V_G$ we will assume that $x$ belongs to the open set $\{G_1(x) \neq 0\}$. All next results do not depend on the particular choice of the open set.
	
	It is well known that $\Sing G \subset M (G)$ and $M(\Psi_G) \subset M(G) \m V_G$. Moreover, we notice that, in the open set $\{G_1(x) \neq 0\}$ one can express the condition $x\in M(G) \m G^{-1}(\Disc G)$ by
	\begin{equation}\label{eq}
	\displaystyle{\nabla \rho(x)=a(x)\nabla \|G(x)\|^{2}+\Sigma_{k=2}^p\alpha_{k}(x)\Omega_{k}(x)},
	\end{equation}
	where $\Omega_k = G_1\nabla G_k - G_k \nabla G_1$, for $k=2,\ldots, p$  are the generators of the normal space $T_{x}X_{y}^{\perp}$ in $\mathbb{R}^{m}$, with $T_{x}X_{y}$ the tangent space of the fiber $X_{y}=\Psi_G^{-1}(y)$, over a regular value $y=\Psi_{G}(x)$. Therefore, one has that for any $x\in M(G)\m G^{-1}(\Disc G)$, $a(x)=0$ if and only if $x\in M(\Psi_G)\m G^{-1}(\Disc G)$. This shows the importance of coefficient $ a(x) $ for study of $ \rho  $-regularity of $ \Psi_G $.  On the other hand, the next theorem characterize the existence of a MVF for $G$ in terms of $ a(x) $. 
	
	\begin{theorem}\cite[Existence of MFV]{ART2}\label{eqt2} Let $G:(\mathbb{R}^{m},0)\to (\mathbb{R}^{p},0),$ $m\geq p \geq 2$. Then the following are equivalent:
		
		\begin{enumerate}
			\item[$(i)$] there exists a MVF for $G$ on $B_{\epsilon}\m G^{-1}(\Disc G),$ for some enough small $\epsilon >0$;
			\item[$(ii)$]  $a(x)>0,$ for any $x\in M(G)\m G^{-1}(\Disc G)$. 
		\end{enumerate}
	\end{theorem}
	
	In order to prove Theorem \ref{eqt2}, the author have considered the following vector fields on $B_{\e}^{m} \m G^{-1}(\Disc G)$: 
	\begin{equation*}\label{eqe6}
	v_{1}(x):=\proj_{T_{x}X_{y}}(\nabla \|G(x)\|^{2}) \textrm{  and  }  v_{2}(x):=\proj_{T_{x}X_{y}}(\nabla \rho(x)).\\
	\end{equation*}
	
	They have noted that vector field $v_{1}$ has no zeros  since the tube $\|G(x)\|^{2}= const.$ is transversal to $X_{y}$, for  $y\not\in \Disc G$ and the vector field $v_{2}$ has no zeros on $B_{\e}^{m}\setminus G^{-1}(\Disc G)$ since $ a(x)>0 $ for any $x\in M(G)\m G^{-1}(\Disc G)$, i.e., $ \Psi_G $ is $ \rho $-regular. Consequently, the vector field 
	\begin{equation}\label{eqe7}
	\nu(x)= \displaystyle{\frac{v_{1}(x)}{\|v_{1}(x)\|}+\frac{v_{2}(x)}{\|v_{2}(x)\|}}
	\end{equation}
	is well-defined on $B_{\e}^{m}\setminus G^{-1}(\Disc G)$
	and has no zero on $ M(G)\m G^{-1}(\Disc G)$. Moreover,  the vectors $v_{1}(x)$ and $v_{2}(x)$ are linearly independent  if, and only if  $x \not\in M(G) \cup G^{-1}(\Disc G)$, (for more details, see \cite{ART2}) hence, $ \nu $  has no singularity on $B_{\epsilon}^m \m  G^{-1}(\Disc G)$. Finally, they gets that the conditions $(c_1)-(c_3)$ are satisfied. Hence, $ \nu $ is a MVF for $G$ on $B_{\epsilon}\m G^{-1}(\Disc G)$.

 Next section we will introduce several sufficient conditions in order to guarantee the existence of such a vector field.

\section{Criteria for MVF existence}\label{cfrho}

Motived by Theorem \ref{eqt2} in this section we give some descriptions of $a(x)$ to ensure the existence of a $ MVF $ for analytic maps $ G $.

\subsection{A Matricial Criterion}
Here we give a matricial criterion for deciding when exists a MVF for a map germ $ G$.

\begin{lemma}\label{l7}
	Let $G:(\mathbb{R}^{m},0)\to (\mathbb{R}^{p},0)$ be smooth map germ. For any $x\in M (G) \m (V_G \cup \Sing G)$ one has that:
	\begin{equation}\label{charac2}
	\displaystyle{a(x)=\frac{\det D(x)}{\det M(x)}} 
	\end{equation}
	where
	
	$$D(x)= \left[ \begin{array}{cccc}
	\langle \nabla \rho(x),  \nabla (\|G(x)\|^{2}) \rangle & \langle \Omega_{2}(x), \nabla (\|G(x)\|^{2}) \rangle  & \cdots & \langle \Omega_{p}(x), \nabla (\|G(x)\|^{2}) \rangle \\
	\langle \nabla \rho(x),  \Omega_{2}(x) \rangle & \langle \Omega_{2}(x), \Omega_{2}(x) \rangle  & \cdots & \langle \Omega_{p}(x), \Omega_{2}(x) \rangle \\
	\vdots & \vdots & \ddots & \vdots \\
	\langle \nabla \rho(x),  \Omega_{p}(x) \rangle & \langle \Omega_{2}(x), \Omega_{p}(x) \rangle  & \cdots & \langle \Omega_{p}(x), \Omega_{p}(x) \rangle
	
	\end{array} \right] $$
	
	\bigskip
	and,
	
	$$M(x)= \left[ \begin{array}{cccc}
	\langle \nabla (\|G(x)\|^{2}),  \nabla (\|G(x)\|^{2}) \rangle & \langle \Omega_{2}(x), \nabla (\|G(x)\|^{2}) \rangle  & \cdots & \langle \Omega_{p}(x), \nabla (\|G(x)\|^{2}) \rangle \\
	\langle \nabla (\|G(x)\|^{2}),  \Omega_{2}(x) \rangle & \langle \Omega_{2}(x), \Omega_{2}(x) \rangle  & \cdots & \langle \Omega_{p}(x), \Omega_{2}(x) \rangle \\
	\vdots & \vdots & \ddots & \vdots \\
	\langle \nabla (\|G(x)\|^{2}),  \Omega_{p}(x) \rangle & \langle \Omega_{2}(x), \Omega_{p}(x) \rangle  & \cdots & \langle \Omega_{p}(x), \Omega_{p}(x) \rangle
	
	\end{array} \right] .\bigskip$$
\end{lemma}

\begin{proof}
	Let  $x\in M (G) \m (V_G \cup \Sing G)$. From equation (\ref{eq}) one gets the matrix equation $\displaystyle{T(x)=M(x)\cdot L(x),}$ where 
	
	\vspace{0.2cm}
	
	$$T(x)=\left[ \begin{array}{c} \langle \nabla \rho(x), \nabla (\|G(x)\|^{2}) \rangle \\  \langle \nabla \rho(x), \Omega_{2}(x) \rangle \\ \vdots \\ \langle \nabla \rho(x), \Omega_{p}(x) \rangle \end{array} \right], $$
	
	\vspace{0.2cm}
	
	$$M(x)= \left[ \begin{array}{cccc}
	\langle \nabla (\|G(x)\|^{2}),  \nabla (\|G(x)\|^{2}) \rangle & \langle \Omega_{2}(x), \nabla (\|G(x)\|^{2}) \rangle  & \cdots & \langle \Omega_{p}(x), \nabla (\|G(x)\|^{2}) \rangle \\
	\langle \nabla (\|G(x)\|^{2}),  \Omega_{2}(x) \rangle & \langle \Omega_{2}(x), \Omega_{2}(x) \rangle  & \cdots & \langle \Omega_{p}(x), \Omega_{2}(x) \rangle \\
	\vdots & \vdots & \ddots & \vdots \\
	\langle \nabla (\|G(x)\|^{2}),  \Omega_{p}(x) \rangle & \langle \Omega_{k}(x), \Omega_{p}(x) \rangle  & \cdots & \langle \Omega_{p}(x), \Omega_{p}(x) \rangle
	
	\end{array} \right] ,$$
	
	\vspace{0.2cm}
	
	$$L(x)= \left[ \begin{array}{c} a(x) \\ \alpha_{2}(x) \\ \vdots \\ \alpha_{k}(x) \end{array} \right].$$
	
	The matrix $M(x)$ is non-singular because $x\notin (V_G \cup\Sing G )$. Moreover by Lagrange's identity its determinant is given by the sum of the squares of all $p\times p$ sub-determinants of the matrix
	
	$$\left[ \begin{array}{c} \nabla (\|G(x)\|^{2}) \\  \Omega_{2}(x)  \\ \vdots \\  \Omega_{p}(x) \end{array} \bigskip \right].$$ \\
	
	Hence, $\det M(x)>0$ and one can write $L(x)=M(x)^{-1}T(x).$  By Cramer's rule the  coefficient $a(x)$ can be written as 
	$$\displaystyle{a(x)=\frac{\det D(x)}{\det M(x)}} $$ \end{proof}

\begin{remark}\label{spr1} It follows form the previous proof that the rank of the matrix $D(x)$ can be understood as the rank of two matrices $A(x)$ and $B(x)$ below considering $D(x)=A(x)\cdot B(x)$:
	
	$$ A(x)= \left[ \begin{array}{c}
	\nabla\|G\|^2 \\
	\Omega_2 \\
	\vdots \\
	\Omega_p
	\end{array} \right]_{p\times m} \text{and} \hspace{1cm} 
	B(x)=\left[\begin{array}{cccc}
	\nabla \rho & \Omega_2 & \cdots & \Omega_p
	\end{array} \right]_{m\times p}.
	$$
	
	Hence, the rank of the matrix $D(x)$ is maximal if, and only if the ranks of the matrices $A(x)$ and $B(x)$ are maximal. Nevertheless, under the condition $x \notin(V_G \cup \Sing G)$ one has that the rank of $D(x)$ is maximal if, and only if the rank of $B(x)$ is maximal. In the case of $ \Disc G = \{0\} $ it   amounts to saying that $x\notin M(\Psi_G)$. 
\end{remark}

\begin{theorem}\label{ll}
Let $G:=(G_1,\ldots,G_p): (\mathbb{R}^m,0) \to (\mathbb{R}^p,0)$ be an analytic map germ. If $\textrm{det}\, D(x)> 0$, then $a(x)>0$ for all $x\in M (G) \m (V_G \cup \Sing G)$. In particular, if $\Disc G=\{0\}$, then there exists a MVF for $ G $ on $ B^{m}_{\e} \m V_{G} $ and $\Psi_G$ is $\rho$-regular.
\end{theorem}

\begin{proof} The proof follows from Lemma \ref{l7} and Theorem \ref{eqt2}. \end{proof}

\begin{remark}
	In \cite{Ma} the author addressed the case $p=2$ and considered $G $ satisfying the ``\textit{Milnor condition $(c)$'' at the origin} if $$\displaystyle{ \|\omega(x)\|^{2}\langle \nabla \rho(x),  \nabla (\|G(x)\|^{2}) \rangle - \langle \omega(x), \nabla (\|G(x)\|^{2}) \rangle \cdot \langle \omega(x), \nabla \rho(x) \rangle}>0,$$ where $\omega(x):=\Omega_{2}(x).$ Using our notations it is equivalent to say that $\det \,D(x)>0$. Thus, our condition ``$\det\, D(x)>0$'' can be thought as a kind of generalization of the Milnor condition $(c)$ at the origin.
\end{remark}

The result below shows that the Milnor tube expands in the radial direction. Hence, since $\Disc G = \{0\}$, in order to inflate the smooth Milnor tube to the sphere\footnote{as was made by Milnor in \cite{Mi},} one can replace the Milnor vector field by the direction $\nabla \|G(x)\|^2$.

\begin{proposition}\label{f}
	Let $ G=(G_1,\ldots, G_p): (\mathbb{R}^m,0) \to (\mathbb{R}^p,0) $ be an analytic map germ, with $ m>p\ge2 $. Then  for any  $x\in B^m_\e \m (V_G \cup \Sing G)$ one has that $ \left\langle \nabla \|G(x)\|^2, \nabla \rho (x) \right\rangle  >0. $
\end{proposition}

\begin{proof} The proof can safely be left as an exercise to the reader. It also can be found in the proof of Theorem 1' by Fukuda \cite{F2}. \end{proof}

As an application of the matricial criterion one has:

\begin{corollary}\label{f1} Let $G: (\mathbb{R}^m,0) \to (\mathbb{R}^p,0)$,  $m>p\ge2$ be an analytic map germ. Suppose that for any $x\in M (G) \m (V_G \cup \Sing G)$ one has either $\left\langle \nabla \|G(x)\|^2, \Omega_{j} (x) \right\rangle = 0$, or  $\left\langle \nabla \rho(x), \Omega_{j} (x) \right\rangle = 0$, for $2\leq j \leq p$. Then $a(x)>0$ for all $x\in M (G) \m (V_G \cup \Sing G)$. In particular, whenever $\Disc G=\{0\}$, there exists a MVF for $ G $ on $ B^{m}_{\e} \m V_{G} $ and $\Psi_G$ is $\rho$-regular.
\end{corollary}

\begin{proof}
	In fact, assume that $ \left\langle \nabla \|G(x)\|^2, \Omega_{j} (x) \right\rangle = 0$ for $2\leq j \leq p.$  It follows from Proposition \ref{f}  that  $\det D(x)= \langle \nabla \rho(x),  \nabla (\|G(x)\|^{2}) \rangle \cdot \det \left ( \langle  \Omega_{i}(x) , \Omega_{j}(x) \rangle_{i,j} \right )>0$.  The case  $\left\langle \nabla \rho, \Omega_{j} (x) \right\rangle = 0$ for $2\leq j \leq p$, follows from the same argument.
\end{proof}

\subsection{Other criteria for the existence of MVF}

Let $G: (\mathbb{R}^{m},0) \to (\mathbb{R}^{p},0)$ be a  analytic map germ  with $G(x)=(G_{1}(x),\ldots, G_{p}(x))$. One can write 

\begin{equation}\label{eq21}
\left\lbrace \begin{array}{cc}
G_1(x)&=G^{1}_{m_{1}}(x)+G^{1}_{m_{1}+1}(x)+ \cdots\\
&\vdots\\
G_p(x)&=G^{p}_{m_{p}}(x)+G^{p}_{m_{p}+1}(x)+ \cdots
\end{array}\right. 
\end{equation} 
where $ G^{i}_{m_{j}} $ is the homogeneous term of degree $m_{j},$ for $ i,j=1,\ldots, p.$ 
We say that $ G $ \textbf{has same multiplicity} if $ m_1 = m_2= \cdots = m_p $.

Our main result is the following.

\begin{theorem}\label{spt1}
	Let $G:=(G_1,\ldots,G_p): (\mathbb{R}^m,0) \to (\mathbb{R}^p,0)$ be an analytic map germ. If $G$ satisfies some of conditions below:
	\begin{enumerate}
		\item[(i)] $G$ has same multiplicity.
		\item[(ii)] $\left\langle \nabla G_i(x), \nabla G_j(x) \right\rangle =0$, for any $i,j=1,\ldots,p$ with $i \neq j$ and $x\in M(G) \m (V_G\cup \Sing G)$.
	\end{enumerate}
	Then  for any $x\in M(G) \m (V_{G}\cup \Sing G)$ the coefficient $a(x)$ is positive. In particular, if $\Disc G =\{0\}$ then there exists a MVF for $ G $ and  $\Psi_G$ is  $\rho$-regular.
\end{theorem}

For the proof of Theorem \ref{spt1}, one needs the matricial criterion and the Lemmas \ref{L1} e \ref{L2} stated below. 

\begin{example}\label{ex-nc} Let $G(x,y,z)=(xy,xz)$.  One can show that $G$ satisfies the condition \eqref{eq:main} and $\Disc G = \{0\}$. Since $ G $ has same multiplicity, for any $x \in M(G)\m V_G$ the coefficient $a(x)$ is positive. Therefore, there exists a MVF for $ G $ and $ \Psi_G $ is $ \rho $-regular. Thus, $G$ has the tube and  sphere fibrations and they are equivalent.
\end{example}

\begin{lemma}\label{L1} Let $G(x)=(G_{1}(x),\ldots, G_{p}(x))$, be an analytic map germ.  For $x\in M(G) \m (V_{G}\cup \Sing G)$ one has
	\[
	\displaystyle{a(x)=\frac{\langle \alpha (x), G(x) \rangle }{\|G(x) \|^2}},
	\]
	where $\alpha(x)=(\alpha_{1}(x),\dots, \alpha_{p}(x)) \in {\bR}^{p}$ is such that $\nabla \rho (x) = \Sigma_{k=1}^p \alpha_{k}(x) \nabla G_{k}(x) $. In particular, $x \in M(\Psi_G) \m (V_G \cup \Sing G) $ if, and  only if $ \left \langle \alpha(x),G(x) \right \rangle=0$. 
\end{lemma}

\begin{proof}
	It follows from Proposition \ref{l7} that, for any $x\in M(G) \m (V_{G}\cup \Sing G)$ the coefficient 
	\[a(x)=\dfrac{\textrm{det}[A(x).B(x)]}{\textrm{det}M(x)} \]
	where the matrices $A(x)$ and $B(x)$ are as in the Remark \ref{spr1} and $ M(x)=A(x).A(x)^{t}.$
	
	\vspace{0.2cm}
	
	Since   $x\notin V_G\cup \Sing G$, one has from Remark \ref{spr1} that $ A(x)$ and $M(x)$ has maximal rank.   
	
	\vspace{0.2cm}
	
	On the open set $\{G_{1}(x)\neq 0\}$ one can rewrite:
	
	$$ A(x)= \left[  \begin{array}{cccc}
	G_{1}(x) & G_{2}(x) & \cdots & G_{p}(x)\\
	-G_2(x) & G_1(x)&\cdots &0 \\
	\vdots &\vdots&\ddots&\vdots\\
	-G_p(x)& 0& \cdots & G_1(x)
	\end{array} \right]\left[\begin{array}{c} 
	\nabla G_1(x) \\
	\nabla G_2(x)\\
	\vdots \\
	\nabla G_p(x)
	
	\end{array} \right] $$
	
	$$ B(x)^{t}= \left[  \begin{array}{cccc}
	\alpha_{1}(x) & \alpha_{2}(x) & \cdots & \alpha_{p}(x)\\
	-G_2(x) & G_1(x)&\cdots &0 \\
	\vdots &\vdots&\ddots&\vdots\\
	-G_p(x)& 0& \cdots & G_1(x)
	\end{array} \right]\left[\begin{array}{c} 
	\nabla G_1(x) \\
	\nabla G_2(x)\\
	\vdots \\
	\nabla G_p(x)
	
	\end{array} \right].\vspace{0.5cm} $$
	Therefore, one has that $ A(x)\cdot B(x)=L_1 (x)\cdot L_2 (x)\cdot  L_3 (x) $ and $ M(x) = L_1 (x)\cdot L_2 (x)\cdot L_1(x)^{t} $, where
	
	$$L_1 (x)= \left[  \begin{array}{cccc}
	G_{1}(x) & G_{2}(x) & \cdots & G_{p}(x)\\
	-G_2(x) & G_1(x)&\cdots &0 \\
	\vdots &\vdots&\ddots&\vdots\\
	-G_p(x)& 0& \cdots & G_1(x)
	\end{array} \right], \vspace{0.5cm} $$
	
	$$L_2(x)=\left[\begin{array}{c} 
	\nabla G_1(x) \\
	\nabla G_2(x)\\
	\vdots \\
	\nabla G_p(x)\end{array} \right] \left[\begin{array}{cccc} 
	\nabla G_1(x) & \nabla G_2(x)& \cdots & \nabla G_p(x)
	\end{array} \right]$$
	the Grassmann matrix of $JG(x),$ and
	
	$$L_3(x)= \left[  \begin{array}{cccc}
	\alpha_{1}(x) & -G_{2}(x) & \cdots & -G_{p}(x)\\
	\alpha_2(x) & G_1(x)&\cdots &0 \\
	\vdots &\vdots&\ddots&\vdots\\
	\alpha_p(x)& 0& \cdots & G_1(x)
	\end{array} \right]. \vspace{0.5cm}$$
	Thus, one has that 
	\[a(x) = \dfrac{\textrm{det} L_3 (x)}{\textrm{det} L_1 (x)}.\]
	
	On the other hand, on the open set $ \{G_1\neq 0\} $ one has that $ \textrm{det} L_1 (x) = (G_1(x))^{p-2}\|G(x)\|^2$. Analogously, applying Laplace's rule in the first column $ \textrm{det} L_3 (x) =(G_1(x))^{p-2} \Sigma_{k=1}^p \alpha_{k}(x)G_{k}(x) $. Therefore, 
	
	\[
	\displaystyle{a(x)=\frac{\Sigma_{k=1}^p \alpha_{k}(x)G_{k}(x)}{\| G(x) \|^2}=\frac{\langle \alpha(x), G(x) \rangle}{\| G(x) \|^2}}.
	\]
\end{proof}

\begin{lemma}\label{L2}
	Let $G(x)=(G_{1}(x),\ldots, G_{p}(x))$ be an analytic map germ.  In the set $M(G) \m (V_{G}\cup \Sing G)$, one has that
	\begin{equation}\label{caseg}
	\left \langle \alpha(x),G(x) \right \rangle = \left \langle (G(x)\cdot[\mathcal{D}(m_{i})\cdot A(x)],G(x) \right \rangle,
	\end{equation}
	where, $ A(x) = [\textrm{J}G(x)\cdot \textrm{J}G(x)^{t}]^{-1} $ and 
	$$ \mathcal{D}(m_{i})= \left[ \begin{array}{cccc}
	m_{1}  & 0 & \cdots & 0 \\
	0 & m_{2} & \cdots &  0  \\
	\vdots & \vdots & \ddots & \vdots \\
	0 & 0  & \cdots &m_{p}
	\end{array} \right] .$$ 
\end{lemma}

\begin{proof} Since $x\in M(G)\m (V_{G} \cup \Sing G)$, by definition there exist $ \alpha_1(x), \ldots, \alpha_p (x) \in \mathbb{R} $ such that
	\[\nabla \rho (x) =\Sigma_{j=1}^{p}\alpha_j (x)\nabla G_j (x).\]
	Let us define $ \alpha(x): = (\alpha_1(x),\ldots, \alpha_p(x)) $. After changing $\alpha(x)$ by $\displaystyle{\frac{\alpha(x)}{2}},$ one can consider the matrix equation 
	$[x]= [\alpha(x)]\cdot [\textrm{J}G(x)]. \medskip$
	Now, multiplying both side of the matrix equation  above  by $ [\textrm{J}G(x)^{t}] $ one gets:
	\begin{equation}\label{eq20}
	[x]\cdot [\textrm{J}G(x)^{t}]= [\alpha(x)]\cdot [\textrm{J}G(x)\cdot \textrm{J}G(x)^{t}] \medskip.
	\end{equation}
	The $ p \times p $ square matrix $[\textrm{J}G(x)\cdot \textrm{J}G(x)^{t}] $ is invertible under the condition $x\notin \Sing (G)\cup V_{G}$. So, after multiplying both side of equality (\ref{eq20}) by its inverse $A(x)=[\textrm{J}G(x)\cdot \textrm{J}G(x)^{t}]^{-1}$, one has the following equation
	$[\alpha(x)]=[x]\cdot [\textrm{J}G(x)^{t}]\cdot A(x).$
	Now, one can write the scalar product (dot product) as 
	\begin{equation}\label{eq22}
	\left \langle \alpha(x),G(x) \right \rangle = \left \langle [x]\cdot [\textrm{J}G(x)^{t}]\cdot A(x),G(x) \right \rangle.
	\end{equation}
	
	It follows from equations in \eqref{eq21} that:
	\[\nabla G_1(x) = \nabla G^{1}_{m_{1}}(x)+\nabla G^{1}_{m_{1}+1}(x)+\cdots  \] 
	\centerline{\vdots}
	\[\nabla G_p(x) = \nabla G^{p}_{m_{p}}(x)+\nabla G^{p}_{m_{p}+1}(x)+\cdots. \]
	Then  by Euler's identity one finds
	$$ \left\langle \nabla G^{i}_{k}(x), x \right\rangle =kG^{i}_{k}(x).$$
	Hence, 
	\[\left\langle \nabla G_i(x), x\right\rangle  = \sum_{j=0}^{\infty}\left\langle \nabla G_{m_{i}+j}^{i}(x),x\right\rangle = \sum_{j=0}^{\infty}(m_{i}+j)G^{i}_{m_{i}+j}(x) = m_{i}G_{i}(x)+ \sum_{j=1}^{\infty}jG^{i}_{m_{i}+j}(x) \] and one can decompose
	\begin{equation} \label{eq-40}
	(\left \langle \nabla G_{1}(x), x \right \rangle, \ldots,  \left \langle \nabla G_{p}(x), x \right \rangle)= 
	(m_{1}G_{1}(x),\ldots, m_{p}G_{p}(x))+\left(\sum_{j=1}^{\infty}jG^{1}_{m_{1}+j}(x), \ldots, \sum_{j=1}^{\infty}jG^{p}_{m_{p}+j}(x)\right).
	\end{equation}
	Denote the vector expression $(\sum_{j=1}^{\infty}jG^{1}_{m_{1}+j}(x), \ldots, \sum_{j=1}^{\infty}jG^{p}_{m_{p}+j}(x))$ by $V(x),$ and denote $\langle V(x)\cdot A(x), G(x) \rangle$ by $H.O.T.$
	
	From equation (\ref{eq22}) one gets $$\left \langle \alpha(x),G(x) \right \rangle = \left \langle (m_{1}G_{1}(x),\ldots, m_{p}G_{p}(x))\cdot A(x),G(x) \right \rangle + H.O.T.$$
	One can now write the vector $(m_{1}G_{1}(x),\ldots, m_{p}G_{p}(x))$ as a matrix product $G(x)\cdot \mathcal{D}(m_{i})$ where $ \mathcal{D}(m_{i})$ is a diagonal matrix with entries $m_{i}$, $1\leq i \leq p$. 	
	Hence, 
	\begin{equation*}
	\left \langle \alpha(x),G(x) \right \rangle = \left \langle (G(x).[\mathcal{D}(m_{i})\cdot A(x)],G(x) \right \rangle +H.O.T.
	\end{equation*} \end{proof}

\begin{proof}[{\bf Proof of Theorem \ref{spt1}}]
	\textbf{item $(i).$}	Since $ G $ has same multiplicity, namely $ m=m_1=\cdots=m_p>0 $ the matrix $ \mathcal{D}(m_i)=m.I_{p\times p} $ where the matrix $ I_{p\times p} $ is the $ p\times p $ identity matrix. Thus,
	$$\left \langle \alpha(x),G(x) \right \rangle = m\left \langle (G(x)\cdot [A(x)],G(x) \right \rangle +H.O.T.$$
	Now, since the symmetric matrix $[\textrm{J}G(x)\cdot \textrm{J}G(x)^{t}] $ is a Gram matrix  and $ \{\nabla G_1(x),\ldots, \nabla G_p (x) \} $ are linearly independent, one has  that $[\textrm{J}G(x) \cdot \textrm{J}G(x)^{t}]$ is positive definite. 
	Consequently,  the inverse matrix $A(x)  $ is positive definite 
	and $m\left \langle (G(x)\cdot A(x),G(x) \right \rangle > 0$. 	Therefore, for any $x\in M(G)\m (V_{G}\cup \Sing G)$ close enough to the origin the scalar product $\left \langle \alpha(x),G(x) \right \rangle$ must be positive.
	
\vspace{0.2cm}
	
	\noindent
	\textbf{item $(ii).$}  Since   $\left\langle \nabla G_i(x), \nabla G_j(x) \right\rangle =0$, for any $i,j=1,\ldots,p$ with $i \neq j$, one has that the positive definite matrix $A(x)^{-1}=[\textrm{J}G(x) \cdot \textrm{J}G(x)^{t}]$ is diagonal. Consequently, the matrix $ \mathcal{D}(m_{i})\cdot A(x)$ must be positive definite and the result follows.
\end{proof}






\section{Applications} 

In this section we show that for some classes of maps there exist the  fibrations \eqref{ssf} and \eqref{sfibration1} and  they are equivalent. 

\begin{definition}\cite{Ma}\label{slm}
	Let $G:=(G_1,\ldots,G_p):(\mathbb{R}^m,0) \to (\mathbb{R}^p,0) $ be an analytic map germ. We say that $G$ is a \textit{Simple \L-Map}, if $\left\langle \nabla G_i, \nabla G_j\right\rangle =0$ for any $i,j=1,\ldots,p$ with $i\neq j$ and $\|\nabla G_i\|=\|\nabla G_j\|$ for any $i,j=1,\ldots,p$. 
\end{definition}

\begin{example} Consider the real map germ $ G:=(G_1,G_2):(\mathbb{R}^8,0) \to (\mathbb{R}^2,0) $ given by 
	$$ \left\lbrace \begin{array}{ccl}
	G_1(x,y,z,w,a,b,c,d)&= -w^2x^2+w^2y^2+4wxyz+x^2z^2-y^2z^2+ac+bd&   \medskip\\
	G_2(x,y,z,w,a,b,c,d)&=-2w^2xy-2wx^2z+2wy^2z+2xyz^2-ad+bc.&  \medskip\\
	\end{array}\right. $$
	
	One can show that $\left\langle \nabla G_1, \nabla G_2 \right\rangle =0$ and $	\|\nabla G_1 \|^2 =  \|\nabla G_2 \|^2$. Thus, $ G $ is a Simple \L-Map.  
\end{example}	

\begin{proposition}\label{c-slm}
	If  $G: (\mathbb{R}^m,0) \to (\mathbb{R}^p,0)$ be an analytic map germ which is a Simple \L - Map.  	Then there exist the fibrations \eqref{ssf} and \eqref{sfibration1} and they are equivalent.
\end{proposition}

\begin{proof} 
If $ G $ is a Simple \L-Map then it satisfies the condition \eqref{eq:main} and $ \Disc G =\{0\} $ by  \cite[Theorem 5.7]{Ma}. Now by Theorem \ref{spt1} item ($ii$) the result follow. \end{proof}

\subsection{Mixed  functions and equivalence}

 Mixed singularities has been systematically studied by Mutsuo Oka in the sequence of papers \cite{Oka2, Oka3, Oka4} published since 2008, thenceforward the term \textit{mixed function} was coined and it has been used to identify function $f=(u+iv):\mathbb{C}^n\to \mathbb{C}$ with $f\left(\textbf{z}, \bar{\textbf{z}} \right)= \Sigma_{\nu, \mu}c_{\nu, \mu} \textbf{z}^{\nu}\bar{\textbf{z}}^{\mu}$, where $c_{\nu, \mu} \in \bC$,  $\textbf{z}^{\nu}$  and $\bar{\textbf{z}}^{\mu}$ are multi-monomials in the variables $z_j,\bar{z}_j$, with $j=1,\ldots,n$.

One may view $f$  as a real analytic map of $2n$ variables $(x,y)$ from $\bR^{2n}$ to $\bR^{2}$ by identifying  $\mathbb{C}^n$ with $\bR^{2n}$, $(z_{1},\ldots, z_{n}) = z\mapsto  (x,y)=(x_1,y_1,\ldots,x_n,y_n)$,  writing $\textbf{z=x+iy} \in \mathbb{C}^{n}$, where $z_{j}=x_{j}+iy_{j} \in \mathbb{C}$ with $x_j, y_j \in \bR$ for $j=1,...,n$. Moreover, one can define the \textit{holomorphic} and \textit{anti-holomorphic} gradients of $ f $ as follows:
\[\textrm{d} f:=\left(\frac{\partial f}{\partial z_{1}},...,\frac{\partial f}{\partial z_{n}} \right) \textrm{  and  } \bar{\textrm{d}}f:= \left(\frac{\partial f}{\partial  \bar{z}_1},...,\frac{\partial f}{\partial \bar{z}_{n}} \right).\]

\vspace{0.2cm}
where, 
\vspace{0.2cm}
\begin{equation*}
\frac{\partial f}{\partial z_j}=\frac{\partial u}{\partial z_j}+i \frac{\partial v}{\partial z_j} \textrm{   and   } \frac{\partial f}{\partial \bar{z}_j}=\frac{\partial u}{\partial \bar{z}_j}+i \frac{\partial v}{\partial \bar{z}_j}.
\end{equation*}

Consequently, one can consider the  identifications
$ \nabla u (x,y)=\overline{\textrm{d} f}(z,\bar{z})+\overline{\textrm{d}}f(z,\bar{z})$ and  $ \nabla v (x,y)=i\left(\overline{\textrm{d} f}(z,\bar{z})-\overline{\textrm{d}}f(z,\bar{z})\right)$. Since $\vec{u}$ and $\vec{v}$ are vectors in $\mathbb{C}^n$ we denote its hermitian product by $\left\langle\vec{u}, \vec{v} \right\rangle_\mathbb{C}$. With this notations and definitions, one has the following.

\begin{corollary}\label{c1}
	A  mixed function germ $f: (\mathbb{C}^m,0) \to (\mathbb{C},0) $ is a Simple \L - Map if and only if $\left\langle \overline{\dd f}, \bar{\dd}f \right\rangle _\mathbb{C}=0$, i.e., $\re\left\langle \overline{\dd f}, \bar{\dd}f \right\rangle _\mathbb{C}=\im\left\langle \overline{\dd f}, \bar{\dd}f \right\rangle _\mathbb{C}=0$. In particular, if $f$ is holomorphic then it is a Simple \L - Map and the fibrations \eqref{ssf} and \eqref{sfibration1} exist and they are equivalent
\end{corollary}

In what follows we will consider an algorithm to construct  classes of mixed functions which are Simple \L-Maps, for short, {\it MSL}. 

\begin{proposition}[Algorithm to MSL]\label{alg}
	Consider the  following algorithm:
	\begin{itemize}
		\item [Step (1).] Fix a copy of $\mathbb{C}^n,$ $n\geq 2$, and a coordinate system $z=(z_{1},\ldots, z_ {n})$. 
		
		\item [Step (2).] For each $1\le k <n$ choose natural numbers $i_1,\ldots,i_k \in \{1,2,\ldots,n\}$ with $i_1<i_2<\ldots<i_k$ and fix the coordinates $(z_{i_{1}},\ldots,z_{i_{k}})$. For the complementary ordered list $q_{1}, \ldots, q_{n-k}\in \{1,2,\ldots,n\}\m\{i_1,\ldots,i_k \}$, $q_{1}< \ldots< q_{n-k}$, consider the reminding coordinates $(z_{q_{1}},\ldots,z_{q_{n-k}})$.
		
		\item [Step (3).] For any natural numbers $j$, $t$ and $p$, choose arbitrary holomorphic functions $f_j(z_{i_{1}},\ldots,z_{i_{k}})$, $r_t(z_{i_{1}},\ldots,z_{i_{k}})$, $g_j(z_{q_{1}},\ldots,z_{q_{n-k}})$ and $h_p(z_{q_{1}},\ldots,z_{q_{n-k}})$.
		\item [Step (4).] Define the mixed function germ $f:(\mathbb{C}^n,0)\to (\mathbb{C},0)$ by $f(z_1,\ldots,z_n)=$
		$$\sum_{\alpha=1}^{j}f_{\alpha}(z_{i_{1}},\ldots,z_{i_{k}})\overline{g_{\alpha}(z_{q_{1}},\ldots,z_{q_{n-k}})}+\sum_{\beta=1}^{t}r_\beta(z_{i_{1}},\ldots,z_{i_{k}})+\sum_{\gamma=1}^{p}\overline{h_\gamma (z_{q_{1}},\ldots,z_{q_{n-k}}) }.$$
	\end{itemize}
	\textbf{Claim:} The mixed function germ $f$ is MSL.
\end{proposition}

\begin{proof}
	By simplicity we will choose $j=t=p=1$, $n=2k$ for some $k\in\mathbb{N}$ and $i_\eta = \eta$ for $1 \le\eta \le k$.
	In this case, 
	
	$$f(z_1,\ldots,z_n)=f_1(z_1,\ldots,z_k)\overline{g_1(z_{k+1},\ldots,z_n)}+r_1(z_1,\ldots,z_k)+\overline{h_1(z_{k+1},\ldots,z_n)}.$$
	We notice that $$\dfrac{\partial f}{\partial z_{\alpha}} = \dfrac{\partial f_1}{\partial z_{\alpha}}g_1+\dfrac{\partial r_1}{\partial z_{\alpha}} \textit{ and } \dfrac{\partial f}{\partial \bar{z}_{\alpha}}=0 $$  
	for   $\alpha = 1,\ldots, k$,
	
	$$\dfrac{\partial f}{\partial z_{\beta}}=0  \textit{  and  } \dfrac{\partial f}{\partial \bar{z}_{\beta}} = \dfrac{\partial g_1}{\partial \bar{z}_{\beta}}f_1+\dfrac{\partial h_1}{\partial \bar{z}_{\beta}}$$ 
	for $\beta = k+1,\ldots,n$. Thus, one has $$\overline{\dd f}=\left(\dfrac{\partial f_1}{\partial z_{1}}g_1+\dfrac{\partial r_1}{\partial z_{1}},\ldots, \dfrac{\partial f_1}{\partial z_{k}}g_1+\dfrac{\partial r_1}{\partial z_{k}},0,\ldots,0\right) $$
	and $$\bar{\dd}f=\left(0,\ldots,0,\dfrac{\partial g_1}{\partial \bar{z}_{k+1}}f_1+\dfrac{\partial h_1}{\partial \bar{z}_{k+1}},\ldots,\dfrac{\partial g_1}{\partial \bar{z}_{n}}f_1+\dfrac{\partial h_1}{\partial \bar{z}_{n}}\right).$$
	Therefore, $\left\langle \overline{\dd f}, \bar{d}f \right\rangle _\mathbb{C}=0$ and by Corollary \ref{c1} the function $f$ is MSL.
\end{proof}

\begin{example} 
	Let $ G:(\mathbb{C}^4,0) \to (\mathbb{C},0) $ given by $ G(z_1,z_2,z_3,z_4)=z_{1}^{2}\bar{z}_{2}^{2}+z_3 \bar{z}_4 +z_{1}^{4}-z_{3} -\bar{z}_2\bar{z}_{4}^{3} $. Considering $f_1(z_1,z_3)=z_{1}^{2}$, $f_2(z_1,z_3)=z_3$, $g_1(z_2,z_4)=z_{2}^{2}$, $g_2(z_2,z_4)=z_4$, $r(z_1,z_3)=z_{1}^{4}-z_{3}$ and $h(z_2,z_4)= \bar{z}_2\bar{z}_{4}^{3}$. Then one has that $G=f_1\bar{g}_1+f_2\bar{g}_2 +r-h$. By Proposition  \ref{alg}, $G$ is a MSL.
\end{example}	

We point out that the Proposition \ref{alg} also extends the Thom-Sebastiani type result obtained in Corollary 4.2 of \cite{PT}. 

\begin{remark} If $f:\mathbb{C}^n\to \mathbb{C} $ is a mixed functions it follows by Theorem \ref{spt1} that only the condition $\im \left\langle \overline{\dd f}(z), \bar{\dd}f(z) \right\rangle_{\mathbb{C}}=0$ is enough to guarantee that the coefficient $a(z)$ is positive. In particular, if $\Disc f=\{0\}$ then there exists a MVF for  $f$ and $ \Psi_f $ is $\rho$-regular.
\end{remark}

\subsection{Product of mixed functions}
Let $F=fg,$ where $ f:\mathbb{C}^n \to \bC $ and $ g:\mathbb{C}^m \to \bC $ are mixed functions in separable variables. If we consider the identifications $ \nabla \|f\|^2  = 2(f\overline{\textrm{d} f}+\bar{f}\bar{\textrm{d}}f) $  and   $ \Omega_f = i(f\overline{\textrm{d} f}-\bar{f}\bar{\textrm{d}}f)$, where $ \Omega_f $ denote the direction $ u \nabla v - v \nabla u $, one can prove that:  

\vspace{0.2cm} 

\noindent $M(F)\m V_F \subset \left(M(f) \m V_f\right)\times \left(M(g) \m V_g \right)$, $M(F/\|F\|) \subset M(f/\|f\|) \times M(g/\|g\|)$.  In particular, if $M(f/\|f\|) =\emptyset$ or $ M(g/\|g\|)= \emptyset $, then $ M (F/\|F\|)  =\emptyset$. 

\vspace{0.2cm} 

In another words, for mixed functions in separable variables in order to get $F/\|F\|$ $\rho $-regular it is enough to ask the $\rho$-regularity for $f/\|f\|$ or $g/\|g\|$. Moreover,  

\vspace{0.2cm} 

\noindent $\Sing F=(\Sing f \times \Sing g) \cup (\left(V_f \cap \Sing f \right) \times \mathbb{C}^m) \cup (\mathbb{C}^n \times \left(V_g \cap \Sing g\right)) \cup (V_f \times V_g).$
Hence, if either $ \Disc  f = \{0\}$, or $ \Disc g= \{0\}$, then $\Disc F =\{0\}$. 

\begin{proposition}\label{eet5}
	Let $ f:\mathbb{C}^n \to \bC $ and $ g:\mathbb{C}^m \to \bC $ be mixed functions in separable variables. Suppose that  the mixed function $ F=f g:\mathbb{C}^n \times \mathbb{C}^m \to \bC $ satisfies the  condition \eqref{eq:main}. If either the conditions below holds true: 
	\begin{enumerate}
		\item [(i)] $\Disc f = \{0\}$ and on $B_{\epsilon_{1}}^{2n}\m V_f$ there exists a MVF for $f$, for some small enough $\epsilon_{1} >0$;
		
		\item[(ii)]$\Disc g = \{0\}$ and on $B_{\epsilon_{2}}^{2m}\m V_g$ there exists a MVF for $g$, for some small enough $\epsilon_{2} >0$,
	\end{enumerate}
	then both fibrations \eqref{ssf} and \eqref{sfibration1} exist and they are equivalent.
\end{proposition}

\begin{proof} We will prove the item $(i)$. The item $(ii)$ follows in the same way.
	By hypothesis $ \Disc F = \{0\} $ thus, for any $ (x,y)\in M(F) \m V_F $ one has	 $ (x,y) \not \in \Sing F. $ Consequently, there exist $ a(x,y), b(x,y)\in \bR $ such that 
	
	\begin{equation*}
	(x,y)=a(x,y)\nabla \|F(x,y)\|^2 +b(x,y) \Omega_F (x,y).
	\end{equation*}
	
	\vspace{0.2cm}
	
	Since, $ \nabla \|F(x,y)\|^2 = \left(\|g(y)\|^2 \nabla\|f(x)\|^2 , \|f(x)\|^2 \nabla \|g(y)\|^2 \right)$ and \\ $ \Omega_F(x,y)=\left(\|g(y)\|^2\,\Omega_f (x),\|f(x)\|^2\, \Omega_g(y)  \right)$, one gets that
	
	\begin{equation}\label{eee5}
	\left\lbrace  \begin{tabular}{l l}
	$  x $&$ = a(x,y)\|g(y)\|^2\nabla\|f(x)\|^2+b(x,y)\|g(y)\|^2 \Omega_f(x) \medskip$ \\
	$  y $ &$ = a(x,y)\|f(x)\|^2\nabla\|g(y)\|^2 + b(x,y)\|f(x)\|^2 \Omega_g(y)  $
	\end{tabular} \right.
	\end{equation}
	
	\vspace{0.2cm}
	
	Hence, $ x\in M(f) \m V_f $ and $ y \in M(g) \m V_g $. 
	
	Suppose that there exists a MVF for $f$.   Let $ (x,y)\in \left(M(F)\m V_F\right) \cap B_{\e_1}^{2(n+m)} $. Since $ x \not \in \Sing f $ there exist $ a_1(x), b_1(x) \in \bR$ such that 
	\begin{equation}\label{eee6}
	x = a_1(x)\nabla \|f(x)\|^2 + b_1(x)\Omega_f(x).
	\end{equation}
	Comparing the first equation of (\ref{eee5}) and the equation (\ref{eee6}), one has that 
	\[\left(a(x,y)\|g(y)\|^2 - a_1(x)\right)\nabla\|f(x)\|^2+\left(b(x,y)\|g(y)\|^2-b_1(x)\right)\Omega_f(x)  =0\]
	Since $ \Disc f = \{0\}$, then $ \left\lbrace \nabla\|f(x)\|^2, \Omega_f(x) \right\rbrace  $  are linearly independent on $ \bR $. Hence, $ a_1(x)=a(x,y)\|g(y)\|^2 $ and $ b_1(x) = b(x,y)\|g(y)\|^2 $. 
	
	\vspace{0.2cm}
	
	By hypothesis, $ a_1(x)>0 $ which implies that $ a(x,y)>0. $ Then, by Theorem \ref{eqt2} the vector field \eqref{eqe7} is a MVF for F. Therefore, by Corollary \ref{tube} and Theorem \ref{sf2} one gets the existence of the Milnor-Hamm tube and sphere fibrations, and by Theorem \ref{tm} they are equivalent.
\end{proof}

\begin{corollary}\label{eec2}
	Let $ f:\mathbb{C}^n \to \bC $ be a holomorphic function and $ g:\mathbb{C}^m \to \bC $ be a mixed functions in separable variables. Suppose that  the mixed function $ F=f g:\mathbb{C}^n \times \mathbb{C}^m \to \bC $ satisfies the  condition \eqref{eq:main}. Then, there exist the fibrations \eqref{ssf} and \eqref{sfibration1} and they are equivalent.
\end{corollary}

\begin{proof}
	
	It is enough to check that $f$ satisfies the condition ($ i $) of Proposition \ref{eet5}. In fact, it follows from Corollary \ref{c1}.
\end{proof}

\begin{example}
	Consider the mixed  function $ F:\mathbb{C}^2 \to \bC$, $ F(x,y)=y\|x\|^2 $. In \cite{ACT} the authors have shown that $ F $ satisfies the condition \eqref{eq:main} but does not have the Thom $ a_F $-condition. Let us consider the mixed functions $ f:\bC \to \bC $ and $ g:\bC \to \bC $ given by $ f(y)=y $ and $ g(x)=x\,\bar{x} $, then $ F=fg $. By Corollary \ref{eec2}, there exist the fibrations \eqref{ssf} and \eqref{sfibration1} and they are equivalent.
\end{example}

\end{document}